\documentclass[12pt]{amsart}
\usepackage{txfonts}
\usepackage{bbm}
\usepackage{mathrsfs}
\usepackage{amssymb}
\usepackage{amssymb,amsmath,amsthm,amscd}
\usepackage[all]{xy}

\usepackage{graphicx}

\numberwithin{equation}{section}

\newtheoremstyle{fancy1}{10pt}{10pt}{\itshape}{12pt}{\textsc\bgroup}{.\egroup}{8pt}{
}
\newtheoremstyle{fancy2}{10pt}{10pt}{}{12pt}{\itshape}{.}{8pt}{ }

\theoremstyle{fancy1}

\newtheorem{cor}[equation]{Corollary}
\newtheorem{lem}[equation]{Lemma}
\newtheorem{prop}[equation]{Proposition}
\newtheorem{thm}[equation]{Theorem}
\newtheorem{problem}[equation]{Problem}
\newtheorem{main}{Theorem}
\newtheorem*{main*}{Theorem}

\newtheorem*{cor*}{Corollary}

\theoremstyle{fancy2}
\newtheorem{definition}[equation]{Definition}
\newtheorem{rem}[equation]{Remark}
\newtheorem*{rem*}{Remark}
\newtheorem{example}[equation]{Example}

\newcommand{\cref}[1]{Corollary~\ref{#1}}

\newcommand{\Sph}{\mathbb{S}}
\newcommand{\Disc}{\mathbb{D}}

\newcommand{\CP}{\mathbb{C\mkern1mu P}}
\newcommand{\HP}{\mathbb{H\mkern1mu P}}
\newcommand{\CaP}{\mathrm{Ca}\mathbb{\mkern1mu P}^2}

\newcommand{\C}{{\mathbb{C}}}
\newcommand{\R}{{\mathbb{R}}}
\newcommand{\Z}{{\mathbb{Z}}}

\newcommand{\G}{\ensuremath{\operatorname{\mathsf{G}}}}

\renewcommand{\S}{\ensuremath{\operatorname{\mathsf{S}}}}

\def\con#1=#2(#3){#1 \equiv #2 \bmod{#3}}

\newcommand{\diam}{\ensuremath{\operatorname{diam}}}
\renewcommand{\sec}{\ensuremath{\operatorname{sec}}}

\newcommand{\no}{\noindent}

\begin{document}

\title{Rigidity Theorems for Submetries in Positive Curvature}

\author{Xiaoyang Chen}
\address{Department of Mathematics\\
University of Notre Dame\\
      Notre Dame, IN 46556\\USA
      }
\email{xchen3@nd.edu}

\author{Karsten Grove}
\address{Department of Mathematics\\
University of Notre Dame\\
      Notre Dame, IN 46556\\USA
      }
\email{kgrove2@nd.edu }

\thanks{The second named author is supported in part by a grant from the National Science Foundation and a Humboldt award.  He also wants to thank the Max Planck Institute and the Hausdorff Center for Mathematics in Bonn for hospitality}

\maketitle

\begin{abstract}
We derive general structure and rigidity theorems for submetries $f: M \to X$, where $M$ is a Riemannian manifold with sectional curvature $\sec M \ge 1$.
 When applied to a non-trivial Riemannian submersion, it follows that $\diam X \leq \pi/2 $. In case of equality,  there is a Riemannian submersion $\Sph  \to M$  from a unit sphere, and
as a consequence, $f$ is known up to metric congruence. A similar rigidity theorem also holds in the general context of Riemannian foliations.
\end{abstract}


\bigskip
The classical so-called Bonnet-Myers theorem asserts that a complete Riemannian $n$-manifold $M$ with sectional curvature, $\sec M \ge 1$ has diameter
 $\diam M \le \pi$. Moreover, if $\diam M = \pi$, $M$ is isomeric to the unit sphere $\Sph^n$ by Toponogov's \emph{diameter rigidity theorem}.
 
\smallskip
The purpose of this note is to analyze and prove analogous rigidity theorems in the general setting of submetries in Riemannian geometry. The example of Riemannian submersions is particularly appealing:

\begin{main} \label{diameter}
Let $f: M \rightarrow N$ be a  (non-trivial) Riemannian submersion with $\sec M \geq 1$. Then the base $N$ has $\diam (N) \leq \frac{\pi}{2}$,
where equality holds if and only if  there is a Riemannian submersion (possibly an isometry) $f_1: \Sph \to M$, where $\Sph$ is a unit sphere.
\end{main}

Here, by non-trivial we of course mean that $f$ is not an isometry. However, our result in particular includes the case of covering maps. For this case we note the interesting fact that there are irreducible  space forms $N^n= \Sph^n/\Gamma_n$ with $\diam N^n$ converging to $\pi/2$ as $n$ goes to infinity \cite{Mc}.

We point out that the conclusion of the theorem yields a \emph{complete metric classification} when the base $N$ has maximal diameter $\pi/2$ (see Corollary \ref{subrig}).
This is because Riemannian submersions from the standard unit sphere were classified in \cite{gromoll1988low} and \cite{wi2001}, and of course $f \circ f_1: \Sph \to N$ is a Riemannian submersion as well.

In particular, if $M$ and $N$ are simply connected, then metrically $f$ is either a Hopf fibration $\Sph^{2n+1} \to \CP^n$, $\Sph^{4n+3} \to \HP^n$, $\Sph^{15} \to \Sph^8(1/2)$, or the induced fibration $\CP^{2n+1} \to \HP^{n}$.

\medskip
It turns out that there is a similar rigidity theorem in the general context  of Riemannian foliations.

\begin{main}
Let $\mathcal{F}$  be a  Riemannian foliation on $M$ with leaves of positive dimensions,  where $\sec M \ge 1$. Then any two leaves are at distance at most $\pi/2$ from one another. Moreover, if equality occurs, there is a Riemannian submersion: $\Sph \to M$.
\end{main}

This again via the recently completed classification of Riemmannian foliations on the unit sphere \cite{gromoll1988low},  \cite{LW}, \cite{wi2001} yields a complete answer in the case of equality (for details see Corollary \ref{folrig}).

We note that with the exception of the Hopf fibration $\Sph^{15} \to \Sph^8(1/2)$ all Riemmannian foliations on $\Sph$ are homogeneous.
It is thus natural to wonder about isometric group action $\G \times M \to M$ with large orbit space, i.e., $\diam M/\G \ge \pi/2$. For this we prove

\begin{main}
Let $\G$ be compact Lie group acting on $M$ isometrically, where $\sec M \ge 1$. If $\diam M/\G \ge \pi/2$, then either

(1) $M$ is a twisted sphere with a suspension action, \ or

(2) There is a Riemannian submersion $\Sph \to M$ and the $\G$ action on $M$ is induced from a reducible action on $\Sph$, \ or

(3) $M$ is the Cayley plane $\Bbb{OP}^2$and $\G$ has an isolated fixed point.
\end{main}

We point out that an isometric $\G$ action on the standard sphere $\Sph$ has orbit space with diameter at least $\pi/2$ if an only if it is a reducible linear action,
 and that no values are taking between $\pi/2$ and $\pi$, and diameter $\pi$ is equivalent to the action having fixed points, i.e., the action is a suspension.
Also, again knowledge of the isometry groups of the standard spaces in question yield complete answers.

\medskip

The different geometric topics described above are tied together through the notion of a \emph{submetry}. These play an important role in Riemannian as well as in  Alexandrov geometry.
There is a wealth of examples on the standard sphere, but a classification is not in sight.
It was proved by Lytchak (Lemma 8.1 in \cite{Ly})  that the radius rad $X$ of the base $X$ of a submetry $f: Y \to X$ where $Y$ is an Alexandrov space with curv$Y \ge 1$ satisfies rad$X \le \pi/2$.
Here we will consider the diameter rather than the radius and manifolds domains in place of Alexandrov spaces.

Our arguments are based on an adaption of the ideas and constructions in \cite{GG}, where a rigidity theorem for manifolds $N$ with $\sec N \ge 1$ and $\diam M = \pi/2$ was obtained.
To keep the exposition tight, familiarity with the methods and results of \cite{GG} and \cite{GS} is expected. For basic facts and tools from Riemannian geometry we refer to \cite{CE75} and \cite{Pet}.

It is our pleasure to thank Alexander Lytchak for helpful and constructive comments including informing us about Boltners thesis \cite{Bo}.

\section{Submetries}	

In this section we set up notation and analyze submetries $f: M \to X$ with $\sec M \ge 1$  and large $X$, i.e., with $\diam X \ge \pi/2$.

\smallskip
We will use the notation $|pq|$ to denote the distance between $p$ and $q$ whether in $M$ or in $X$. For any closed subset $L$ of either $M$ or $X$, and $r > 0$,  we let

$$ B(L,r) := \{ x \ | \  |x L|  <  r \ \}, $$

\no be the open $r$-neighborhood of $L$, and

$$C(L,r) := \{ x \ | \  |x L|  \ge r \ \}$$

\no its complement.

Recall that by definition, $f$ is a \emph{submetry} if and only if $f(B(p,r)) = B(f(p),r)$ for all $p \in M$ and all $r > 0$.
It is well known and clear from the $4$-point characterization of a lower curvature bound that $X$ is an Alexandrov space with curv$X \ge 1$ (see \cite{BGP}).
\smallskip

Since obviously $\diam M \ge \diam X$, our investigations will be devided into the cases: (1) $\diam X = \pi$, (2) $\pi/2 < \diam X < \pi$, (3) $\pi/2 = \diam X < \diam M$, and (4) $\diam X = \diam M = \pi/2$.

\smallskip

The core of our arguments are convexity and critical point theory for non-smooth distance functions. In fact, from distance comparison with the unit 2-sphere, note that

\begin{lem}  \label{convexset}
For any closed set $L$, and any $r \ge \pi/2$ the set  $C(L,r)$ is  locally convex, and (globally) convex except for the case, where $r=\pi/2$ and $C(L,\pi/2)$  contains two points at distance  $\pi$.
 Moreover,  if $L \subset X$  and $K = f^{-1}(L)$ then $C(K,r) = f^{-1}(C(L,r))$.
\end{lem}

 The latter claim follows since $f$ is a submetry.  It will be important for us, that if $C(L,r)$ has non-empty boundary, then there is a unique point at maximal distance to the boundary, called its \emph{soul point}.
It follows that  $C(L,r)$ is contractible and if it is a subset of $M$,  then it and its small metric tubes are topologically  discs by Lemma 2.6 in \cite{GG}.

\smallskip

From Toponogov's \emph{maximal diameter theorem} and its analogue for Alexandrov spaces (cf., e.g., \cite{GM} remark 2.5) we have the following solution to case (1) above:

\begin{prop}[Metric Suspension]\label{metsusp}
Suppose $|xy| = \diam X = \pi$ and $Y = \{ z \in X \ | |zx| = |zy| = \pi/2\}$. Then $X = \Sigma Y$ the spherical suspension of $Y$, $M = \Sph = \Sigma E$ is a unit sphere with equator $E$, and $f$ is the suspension of its restriction $f_{|E}: E \to Y$.
\end{prop}

\begin{rem}
The classification of general submetries from the unit sphere is an open and important problem.
\end{rem}


Now suppose $x, y \in X$ are at maximal distance in $X$, and $\pi/2 < |xy| = \diam X < \pi$. Then by distance comparison $y$ is uniquely determined by $x$ and vice versa. Then for $K = f^{-1}(x)$,
 $C(K,\pi/2)$ has nonempty boundary consisting of points at distance exactly $\pi/2$ to $K$. Moreover, $C(K,|xy|) = f^{-1}(y)$ is in the interior of the convex set $C(K,\pi/2)$.
 Clearly, $f^{-1}(y)$ is  at maximal distance to the boundary of $C(K,\pi/2)$, i.e., $ f^{-1}(y)$ is a soul point $q \in M$.
 Reversing the roles of $x$ and $y$ we see that  $f^{-1}(x) = p \in M$ is a point as well.

From critical point theory for $|x \cdot|$ in $X$ and $|p \cdot|$ in $M$,  it follows that $X$, respectively $M$, topologically is the suspension of its space of directions, $\Sph_xX$, respectively $\Sph_pM$, the latter being a unit sphere.
Moreover, $f$ induces a submetry $F: \Sph_p M \to \Sph_xX$ via its \emph{differential}, $Df: T_pM \to T_xX$, \cite{Ly}, and topologically $f$ is the suspension of $F$.

In all, we have proved

\begin{prop}[Topological Suspension]\label{topsusp}
Suppose $|xy| = \diam X > \pi/2$. Then $f^{-1}(x)= p$ and  $f^{-1}(y) = q$ are points in $M$. Moreover, topologically $X$ is the suspension of an Alexandrov space Y,
with $curv Y \ge 1$, $M$ is a sphere and $f$ is the suspension of a submetry $F: \Sph^k \to Y$,
where $k=dim M-1$.
\end{prop}

We now proceed to the cases where $\diam X = \pi/2$. Of course, in these cases $r = \pi/2$ for all convex sets $C(L,r)$ we consider. In fact, we now fix  $x, y \in X$ at maximal distance in $X$, i.e., $|xy| = \diam X$, and consider the convex sets

 $$B = \{ z \in X \ | \ |zx| \ge \pi/2 \} = C(x,\pi/2) \ ,  \  B' = \{ z \in X \ | \ |zB| \ge \pi/2 \} = C(B,\pi/2)$$

\no respectively

$$A = \{ p \in M \ | \ |p f^{-1}(x)| \ge \pi/2 \}= f^{-1}(B) \ ,  \  A' = \{ p \in M \ | \ |pA| \ge \pi/2 \}= f^{-1}(B')$$

\no Note, that in fact, the inequalities $\ge \pi/2$ in the definitions of $B$, $A$, etc., can now be replaced by equalities.

\smallskip

It is important to note, that from critical point theory applied to either of the non-smooth distance functions $|A  \cdot|$ or $|A'  \cdot|$, it follows that

\begin{equation}\label{decomp}
M = \Disc (A)  \cup  \Disc(A'),
\end{equation}

\no where $\Disc (A)$ and  $\Disc (A')$ are closed distance tubes around $A$ and $A'$ respectively. As  mentioned above, if say $A'$ has non-empty boundary, it follows that $\Disc(A')$ topologically is a disc.
\medskip

There are now two scenarios depending on whether or not $A$ and or $A'$ has non-empty boundary. Whether or not $\diam M > \pi/2$, the above decomposition of $M$ means that the arguments used in Propositions 3.4, 3.5 in
\cite{GG}
carries over verbatim, yielding

\begin{lem}
The dual sets $A$ and $A'$ either both have empty boundary, or they both have non-empty boundary.
\end{lem}

In the non-rigid case where both $A$ and $A'$ have non-empty boundary and hence $M$ topologically is a sphere, we can complete the investigation along the lines above, as long as the restricted submetries

$$ f: A \to B   \quad  and  \quad  f: A' \to B' $$

\no satisfy the following natural condition.

\begin{definition}
A submetry $f: A \to B$ \emph{respects the boundary} if whenever $f^{-1}(z) \cap \partial A \ne \emptyset$ then $f^{-1}(z) \subset \partial A$.
\end{definition}

\begin{rem} \label{soulpoint}
Note that if $f: A \to B$ respects the boundary then the distance function $|\partial A \cdot|$ is constant along any "fiber"$f^{-1}(x)$ in the interior of $A$.
\end{rem}

As we will see in the next section, there are several natural geometric situations where indeed $f$ respects the boundary. However, here are some examples where the boundary is not respected.

\begin{example}
1. Let $A$ be a hemisphere and $p$ a point on the boundary of $A$.  Let $f: A \to [0,\pi] = B$ be the submetry $f:= |p \cdot|$, i.e.,  $f^{-1}(t)=\{x\in A \ |  \ |xp|=t \}, 0 \leq t \leq \pi $.
Then for any $0<t<\pi$, $f^{-1}(t) \cap \partial A \ne \emptyset$ but $ f^{-1}(t)$ is not contained in $\partial A$. Note that $f^{-1}(t)$ has non-empty boundary for $0<t<\pi$.


2. In this \emph{tennis ball example}, let $A \subset \Sph^2$ be a geodesic of length $\pi$ , and $f: \Sph^2 \to [0, \pi/2]$ again the submetry $f:= |A \cdot|$, i.e.,  $f^{-1}(t)=\{x\in \Sph^2 \ |  \ |x A|=t \}, 0 \leq t \leq \pi/2 $.
Then $A'$ is a geodesic arc opposite to $A$, $B$ is the left end point of $[0,\pi/2]$ and $B'$ the right end point.
For the restricted submetry $f: A \to B$,  $A$ itself is a fiber having non-empty boundary. Of course $A \cap \partial A \ne \emptyset$
but $A$ is not contained in $\partial A$.

\end{example}

\begin{rem}If $A$ and $A'$ have non-empty boundary (allowing points) which is respected by $f$,  then by Remark \ref{soulpoint}, their soul points are the inverse images of the soul points in $B$ and $B'$
for the distance functions to $f(\partial A) \subset \partial B$ and to $f(\partial A') \subset \partial B'$ and one recovers the exact same structure as in Proposition \ref{topsusp} above. The tennis ball example above shows that this is not the case if  $f$ does not respect the boundary.
\end{rem}

 \smallskip

In the remaining case where $A$ and $A'$ are smooth totally geodesic submanifolds we have

\begin{prop}[Rigidity]\label{subrigidity}
Assume $\diam X = \pi/2$ and the pair of dual sets $A, A'$  have no boundary. Then either

(1) There is a Riemanian submersion $F: \Sph \to M$ (possibly an isometry) and moreover $F^{-1}(A)$ and $F^{-1}(A')$ is a pair of totally geodesic dual sub spheres in $\Sph = F^{-1}(A) * F^{-1}(A')$.

\no or

(2) $M$ is isometric to the Cayley plane $\Bbb{OP}^2$, and one of the fibers of $f$ is a point.
\end{prop}

\begin{proof}
If $\diam M = \pi/2$ this is a direct consequence of the diameter rigidity theorem of \cite{GG}, and \cite{wi2001}.

No assume $\diam M > \pi/2$, and thus $M$ by \cite{GS} is a sphere topologically. In this case we need to show that $A$ and $A'$ are dual spheres with diameter $\pi$. This is immediately obvious in the exceptional case of
Lemma \ref{convexset},  where $A$ contains two points at distance $\pi$, i.e., $M$ is the unit sphere.

Now suppose, that both $A$ and $A'$ are totally geodesic submanifolds of dimensions $a \ge 1$, and $a' \ge 1$.  Arguing exactly as Propositions 3.4, 3.5 and Theorem 3.6 in \cite{GG},
one shows that any unit normal vector to either of $A$ and $A'$ defines a minimal geodesic to the other, and that the corresponding maps $\Sph_p^{\perp} \to A'$, $p \in A$ and $\Sph_{p'}^{\perp} \to A$, $p' \in A'$
are Riemannian submersions, where $\Sph_p^{\perp}$ is the unit normal sphere of $A$ at $p$. We are going to show that $A$ is a sphere with diameter $\pi$  and so is $A'$ by similar arguments.
 By transversality and \ref{decomp}, we see that
$$(i_A)_*: \pi_k(A)\rightarrow \pi_k(M)$$
is injective for $k \leq dim(M)-a'-2$ and surjective for $k \leq dim(M)-a'-1$, where $i_A: A \rightarrow M$  is the inclusion map.
It follows that $a>dim(M)-a'-2$ since $A$ is a closed submanifold of $M$ and $M$ topologically is a sphere.
On the other hand, since $A$ and $A'$ are disjoint totally geodesic submanifolds of $M$, we have $a+a' \leq dim(M)-1$ by Frankel's Theorem \cite{Frankel}.
Thus $a+a'=dim(M)-1$ and hence $dim(S_{p'}^{\perp})=dim(A)$. It follows that the Riemannian submersion  $\Sph_{p'}^{\perp} \to A$ is a covering map.

By the arguments in the proof of Lemma 4.1 and Proposition 4.2 in \cite{GG}, we see that $A$ is either simply connected
or a closed geodesic of length $2\pi$.  In the simply connected case,  $A$ is a
constant curvature $1$ sub-sphere since  $\Sph_{p'}^{\perp} \to A$ is a covering map. Because $A$ is also convex, the diameter of $M$ is $\pi$ and we are done.
Otherwise $A$ is a closed geodesic of length $2\pi$ and again  $diam (M)=\pi$ since $A$ is convex.
\end{proof}

\begin{rem}[Fat joins]\label{fat joins}
Since $f \circ F: \Sph \to X$ is a submetry, to understand the possible submetries occurring in (1) above, it suffices to describe submetries $f: \Sph = A * A' \to X$ where of course $f_{|A}:A \to B$ and $f_{|A'}:A' \to B'$ are submetries from standard dual sub spheres of $\Sph$. For this we note that any point of $x\in B$ is joined to any point $x' \in B'$ by a minimal geodesic of length $\pi/2$. Moreover, for any direction $\xi \in S_x^{\perp}$ orthogonal to $B$ there is a minimal geodesic from $x$ to $B'$ with direction $\xi$, and the maps 

$$P_x: S_x^{\perp} \to B', x \in B, \  \  \text{and similarly } \  \  P_{x'}: S_{x'}^{\perp} \to B, x' \in B'$$

\no are submetries. Now pick points $p \in f^{-1}(x)$ and $p' \in f^{-1}(x')$ and note that the submetry $\Sph_p^{\perp} \to B'$ which is the composition  of the isometry  $\Sph_p^{\perp} \to A'$ with $f_{| A'}: \Sph_p^{\perp} \equiv A' \to B'$ also can be written as $P_x \circ Df_p: \Sph_p^{\perp} \to B'$ and similarly for $p'$. These data completely describe $f$.

As an example of the description provided above take $f_{| A}: A = \Sph^{2m+1} \to \CP^m = B$ and $f_{| A'}: A' = \Sph^{4n+3} \to \HP^n = B'$ to be Hopf maps, and let $S_x^{\perp} = \CP^{2n+1}$ and $S_{x'}^{\perp} = \Sph^{2m+1}$.  
\end{rem}

From the above remark it follows that the rigid submetries from part (1) of the above proposition, can be though of as ``fat joins" of the restricted submetries $f_{| A}$ and $f_{| A'}$

\section{Submersions, Foliations and Group Actions}  

A common feature of submetries $f: M \to X$ arising from Riemannian submersions, isometric group actions, or more generally from (singular) Riemanninan foliations is
 that all ``fibers", $f^{-1}(x)$, $x \in X$ are (smooth) closed submanifolds of $M$ (allowing a discrete set of points)
and every geodesic on $M$ that is perpendicular at one point to a fiber remains perpendicular to every fiber it meets. 
We will refer to such a submetry as a \emph{manifold submetry} (This is actually a special case of the notion of \emph{splitting submetries} introduced by Lytchack in his thesis \cite{Ly}).

The following result due in a special case to Boltner \cite{Bo}, lemma 2.4 (cf. also \cite{FGLT} lemma 4.1) allows us to invoke all results from section 1.

\begin{lem}
Let $f: M \to X$ be a manifold submetry between a complete Riemannian manifold $M$ and an  Alexandrov space $X$.
Then $f$ respects the boundary $\partial C$ of any closed convex subset  $C \subset M $ saturated by the fibers over $f(C)$, i.e., with $C = f^{-1}(f(C))$. 
\end{lem}

\begin{proof}

Suppose there is a fiber $f^{-1}(x), x\in X$ such that $f^{-1}(x) \cap \partial C \ne \emptyset$,  we want to show  $f^{-1}(x) \subset \partial C$.
If not, there are two points $p,q \in f^{-1}(x)$ such that $p\in C^{\circ}$, the interior of $C$, and $q\in \partial C$. 
 Since $f^{-1}(x)$ is a smooth closed submanifold and $C$ is saturated by the fibers of $f$, 
the tangent space of $f^{-1}(x)$ at $q$ is tangent to the boundary $\partial C$. Then by the convexity of $C$, there is a  geodesic $\gamma$ in $M$ passing through
$q$, perpendicular to $f^{-1}(x)$ such that $\gamma(0)=q$ and $q_1:=\gamma(\epsilon) \in C^{\circ}$, $q_2:=\gamma(-\epsilon)\notin C$ for $\epsilon>0$
sufficiently small. By Proposition 1.17 in \cite{Bo}, $f \circ \gamma $ can be lifted to a geodesic
$\tilde{\gamma}$ at $p$. Let $p_1:=\tilde{\gamma}(\epsilon), p_2:=\tilde{\gamma}(-\epsilon)$. Since $\gamma_{|[0,\epsilon]} \subset C$, so is $\tilde{\gamma}_{|[0,\epsilon]}$
as $C$ is saturated by the fibers over $f(C)$. Then as $p \in C^{\circ}$, by taking $\epsilon$ small enough, we see that $p_2= \tilde{\gamma}(-\epsilon) \in C^{\circ} \subset C$.
This, however, is impossible since  $q_2$ and $p_2$ are in the same fiber and $C$ is saturated by the fibers over $f(C)$.
\end{proof}

Using this we have the following rigidity theorem:

\begin{thm}[Submetry rigidity]\label{maintheorem} 
Let $M$ be a Riemannian $n+1$- manifold with $\sec M \ge 1$, and $f: M \to X$ a manifold submetry with $\diam X \ge \pi/2$. Then one of the following holds:

(1) There is a manifold submetry $F: \Sph^n \to Y$ such that topologically, $f$ is the suspension of $F$, i.e., up to homeomorphism $f= \Sigma F: \Sigma \Sph^n \to \Sigma Y = X$, \  

(2) There are manifold submetries $f_{\pm}: \Sph_{\pm} \to X_{\pm}$,  such that either

\  \ (a) $M = \Sph_{-} * \Sph_{+}$ up to isometry and $f_{| \Sph_{\pm}} = f_{\pm}$,  or

\  \ (b) There is a Riemanian submersion $f_1: \Sph_{-} * \Sph_{+}  \to M$ such that $f \circ f_1$ is as in case (a), 

\no or

(3) $M$ is isometric to the Cayley plane $\Bbb{OP}^2$, and $f$ has a point fiber.
\end{thm}

\begin{rem}
In the above Theorem, we point out that $f$ possibly could be a submetry with all fibers a finite set of points.

In case (2) the submetries are determined as in Remark \ref{fat joins}.

In case (3), let $p = A$ be a point fiber of $f$ and $A'$ its dual set $\Sph^8(1/2)$. Clearly $f$ restricted to a small convex ball around $p$ is a submetry, as is its restriction to its boundary with the induced length metric metric. The latter in turn is equivalent to the submetry $\Sph^{15} \to \Sph_{f(p)}X$ given by the differential of $f$ at $p$ .
Moreover, its fibers are differentiably equivalent to the fibers of $f$ over $X - (f(p) \cup B')$,
which in turn induce the submetry $f_{| A'}: \Sph^8(1/2) \to B'$. In other words, $f: \CaP  \to X$
is described via a manifolds submetry $F: \Sph^{15} \to Y$, which induces a manifold submetry $\Sph^8(1/2) \to B'$ via the Hopf fibration .

The complete list of manfolds $M$ with $\sec M \ge 1$ and $\diam M = \pi/2$ occurring in (b) consists of $\HP^n$, $\CP^n$, $\CP^{odd}/\Z_2$, and all \emph{reducible space forms},
i.e., $\Sph^n/\Gamma$, where the free isometric action by $\Gamma$ is reducible as a representation on $\R^{n+1}$.
\end{rem}


We now restrict attention to geometric objects of common interest where the above theorem will yield even more information.

The simplest is of course that of \emph{Riemannian submersions} $f: M \to N$, which in particular are locally trivial bundles with manifold fiber.
This obviously rules out the cases (1) and (3) in Theorem \ref{maintheorem}, since  $f$ does not have any point fibers. In particular, we have $\diam N \le \pi/2$.
What is left is the following analogue of Toponogov's maximal diameter theorem.

\begin{cor}[Submersion rigidity]\label{subrig} 
Let $M$ be a Riemannian manifold with $\sec M \ge 1$. Then for any non-trivial Riemannian submersion $f: M \to N$, we have $\diam N \le \pi/2$. Moreover, in case of equality, the following leads to an exhaustive list up to metric congruence:

(1)  Hopf fibrations $\Sph^{kn+ k-1} \to \mathbb{P}^n(k)$ and its induced fibration  $\mathbb{P}^{2n+1}(\C) \to \mathbb{P}^n(\mathbb{H})$,  where $k = \C, \mathbb{H}$, or its real dimension.

(2) The $\mathbb{Z}_2$ cover, $\mathbb{P}^{2n+1}(\C) \to \mathbb{P}^{2n+1}(\C)/\Z_2$,

(3) The covering maps $\Sph^n \to \Sph^n/\Gamma$, where $\Gamma$ acts reducibly on $\R^{n+1}$.

\no In addition some of these maps can be composed or factored yielding others.
\end{cor}

\begin{rem}
As examples of what the possible operations in the above theorem lead to, we point out, e.g.,  $\Sph^{4n+3} \to \mathbb{P}^{2n+1}(\C)/\Z_2 $ and $ \Sph^{2n+1}/\Gamma \to \mathbb{P}^n(\mathbb{C})$ if $\Gamma  = \Z_m \subset \S^1$, $\S^1$ the Hopf circle.
\end{rem}


We next move to \emph{Riemannian foliations}, $\mathcal{F}$ on $M$. Here not all leaves need to be closed,
 so we consider the decomposition $\bar{\mathcal{F}}$ of $M$ by the closures of the leaves in $\mathcal{F}$.
It is well known \cite{Mo} that this yields a manifold submetry $f: M \to M/\bar{\mathcal{F}} =: X$.
 Again, when all leaves have non-zero dimension the cases (1) and (3) in Theorem \ref{maintheorem} are excluded, and in particular $\diam X \le \pi/2$ for this case as well. Specifically we have, 
 
\begin{cor}[Foliation rigidity]\label{folrig}
Let $M$ be a Riemannian manifold with $\sec M \ge 1$ and $\mathcal{F}$ a Riemannian foliation on $M$ with leaves of dimension $k \ge 1$. Then any two leaves in $M$ are at distance at most $\pi/2$ from one another.
Moreover, in case of equality we have the following up to metric congruence:

(1) $M$ is $\Sph^n$, and $\mathcal{F}$ is given by an almost free isometric action by $\mathsf{R}$ , $\S^1$ or $\S^3$, or $\mathcal{F}$ is the Hopf fibration $\Sph^{15} \to \Sph^8(1/2)$,

(2) $M$ is $\mathbb{P}^{}(\C)$ and $\mathcal{F}$ is induced from a $3$-dimensional foliation from (1). 

\end{cor}

The proof is a direct consequence of Theorem \ref{maintheorem} and the classification of Riemannian foliations on the unit sphere \cite{gromoll1988low},  \cite{LW} and \cite{wi2001}.
The only additional input needed, is the obvious fact that any Riemannian foliation $\mathcal{F}$ on the base $N$ of a Riemannian submersion $M \to N$ lifts to $M$ with leaf dimension increasing by the dimension of the fiber.

\vskip0.1in

For general singular Riemannian foliations $\mathcal{F}$ on $M$, it is of course limited what we can say. In fact, even on the unit sphere there are  now numerous examples known \cite{Ra},
and one is far from a classification. Nevertheless, our results show that 

\begin{cor}[Singular foliation rigidity]\label{sing fol}
Let $M$ be a Riemannian manifold with $\sec M \ge 1$ and $\mathcal{F}$ a singular Riemannian foliation on $M$ having leaves at distance at least $\pi/2$. Then either

(1) $M$ is a twisted sphere and $\mathcal{F}$ has two point leaves at maximal distance,

(2) There is a Riemanian submersion $\Sph \to M$,

\no or

(3) $\mathcal{F}$ is a singular Riemannian foliation on $\Bbb{OP}^2$ with a point leaf.
\end{cor}

\vskip0.1in

We conclude our note by a description of the more special \emph{homogeneous} case, i.e., when our manifold submetry $f: M \to M/\G = : X$ is the orbit map for an isometric action by a compact Lie group $\G$ on $M$.
As in the case of singular Riemannian foliations, the diameter of $M/\G$ can take any value up to $\pi$.
This is in contrast to the case of actions on the standard sphere $\Sph$, where simple convexity arguments as already mentioned yield the following:
 The orbit space $\Sph/\G$ has diameter $\ge \pi/2$ if and only if it is a reducible linear action. No values are taking between $\pi/2$ and $\pi$,
and diameter $\pi$ is equivalent to the action having fixed points, i.e., the action is a suspension.

In general we have

\begin{cor}[Group action Rigidity]
Let $\G$ be compact Lie group acting on $M^n$ isometrically, where $\sec M \ge 1$. If $\diam M/\G \ge \pi/2$, then either

(1) $M$ is a twisted sphere with action topologically the suspension of a linear action on $\Sph^{n-1}$,

(2) $M$ is isometric to either $\Sph^n$ or to the base of a Riemannian submersion $\Sph^n \to N$, and the action on $N$ is induced from a reducible action of $\G$ on $\Sph^n$,   or

(3) $M$ is the Cayley plane $\Bbb{OP}^2$ and $G$ has an isolated fixed point.
\end{cor}

\begin{rem}
Since the full isometry group is known for all spaces in (2) and (3), it is possible to get exhaustive and complete statements in those cases.

In case (1) $\G$ has two fixed points at maximal distance $\diam M = \diam M/\G \ge \pi/2$, and from critical point theory there is a smooth $\G$ invariant gradient like vector field on $M$
which is radial near the fixed points at maximal distance. This means that $M = \Disc^n \cup_f \Disc^n$, where $f: \Sph^{n-1} \to \Sph^{n-1}$ is a $\G$ invariant diffeomorphism,
where the action of $\G$ is the isotropy action at a fixed point.
\end{rem}

Of course, the positively curved twisted sphere $M$ above can potentially be exotic, namely if $f$ does not extend to the disc $\Disc^n$. We note that if $\G$ acts
transitively on $\Sph^{n-1}$, then $f$ is actually linear (cf. Lemma 2.6 of \cite{GrS}) and hence it extends. Also in \cite{GK} it was proved that if $\G$ acts by cohomogeneity one on $\Sph^{n-1}$, then $f$
 is isotopic to a linear map and hence again extends. For much smaller actions, we of course expect there to be exotic $\G$ gluing diffeomorphism.
However, at the moment we do not know of any theory that addresses the following problem:

\begin{problem}
Given a linear $\G$ action on $\Sph^{n-1}$. Which properties of the action prevents the existence of exotic $\G$ equivariant diffeomorphisms.
\end{problem}

{}


\begin{thebibliography}{}

\bibitem{Bo} C. Boltner, \emph{On the Structure of Equidistant Foliations of Euclidean Space}, Ph.D.Thesis, Augsburg 2007, arXiv:0712.0245v1.

\bibitem{BGP}  Y. Burago, M. Gromov and G. Perelmann,
\emph{A.D. Alexandrov spaces with curvatures bounded below}, Russian Math. Surveys, \textbf{47} (1992), 1--58.

\bibitem{CE75} J. Cheeger and D. G. Ebin. Comparison theorems in Riemannian geometry. \textbf{365}. \emph{American Mathematical Soc.}, 1975.


\bibitem {FGLT} L. Florit, O. Goertsches, A. Lytchak, and D. T\"{o}ben, \emph{Riemannian foliations on contractible manifolds}, preprint.

\bibitem{Frankel} T. Frankel,  \emph{Manifolds with positive curvature}. Pacific J. Math. \textbf{11} (1961), 165--174.

\bibitem {GG} D. Gromoll and K. Grove,  \emph{A generalization of Berger's rigidity theorem for positively curved manifolds}, Annales scientifiques de l'\'{E}cole Normale Sup\'{e}rieure \textbf{20} (1987),  227--239.

\bibitem {gromoll1988low} D. Gromoll and K. Grove, \emph{The low-dimensional metric foliations of Euclidean spheres}, J. Differential Geom. \textbf{28} (1988),  143--156.

\bibitem{GK} K. Grove and Kim, \emph{Positively curved manifolds with low fixed point cohomogeneity}, J. Differential Geom. \textbf{67} (2004), 1--33.

\bibitem{GM} K. Grove and S.  Markvorsen, \emph{New extremal problems for the Riemannian recognition program via Alexandrov geometry}, J. Amer. Math. Soc. \textbf{8} (1995), 1-28.

\bibitem{GrS} K. Grove and C. Searle, \emph{Differential topological restrictions by curvature and symmetry}, J. Differential Geom. \textbf{47} (1997), 530-559. 

\bibitem{GS} K. Grove and K. Shiohama,  \emph{A generalized sphere theorem}, Ann. of Math.  \textbf{106} (1977), 201--211.

\bibitem{Ly} A. Lytchak , \emph{Submetrien von Alexandrov Raumen}, Ph. D. Thesis, Bonn,

\bibitem{LW} A. Lytchak and B. Wilking, \emph{Riemanninan foliations in spheres}, preprint.

\bibitem{Mc} J. McGowan, \emph{The diameter function on the space of space forms}, Compositio Mathematica, \textbf{87}, (1993), 79--98.

\bibitem {Mo} P. Molino, Riemannian foliations, \emph{Birkha\"{u}ser} Boston Inc., Boston MA, 1988.


\bibitem {Pet} P. Petersen, Riemannian geometry. Second edition. \emph{Graduate Texts in Mathematics}, \textbf{171}. Springer, New York, 2006.

\bibitem{Ra} M. Radeschi, \emph{Clifford algebras and new singular Riemannian foliations in spheres}, preprint.

\bibitem{wi2001} B.Wilking, \emph{Index parity of closed geodesics and rigidity of Hopf fibrations}. Invent. Math. \textbf{144} (2001), 281--295.

\end{thebibliography}
\end{document}